\documentclass{article}
\usepackage[final]{graphicx}
\usepackage{psfrag}
\usepackage{amsmath,amsfonts,amssymb,amsxtra,subeqnarray}

\newcommand {\Real}{\ensuremath{{\mathbb{R}}}}
\newcommand {\Natural}{\ensuremath{{\mathbb{N}}}}

\newcommand {\Rational}{\ensuremath{{\mathbb{Q}}}}

\newcommand{\X}{\ensuremath{\mathcal X}}
\newcommand{\Y}{\ensuremath{\mathcal Y}}

\newcommand{\N}{\ensuremath{\mathcal N}}

\newcommand{\vi}{\ensuremath{{\mathbf{v}}}}
\newcommand{\ex}{\ensuremath{{\mathbf{x}}}}
\newcommand{\zi}{\ensuremath{{\mathbf{z}}}}

\newcommand{\vay}{\ensuremath{{\mathbf{y}}}}
\newcommand{\one}{\ensuremath{{\mathbf{1}}}}

\newcommand{\xhat}{\hat x}

\newtheorem{theorem}{Theorem}

\newtheorem{lemma}{Lemma}
\newtheorem{claim}{Claim}

\newtheorem{definition}{Definition}
\newtheorem{remark}{Remark}
\newtheorem{assumption}{Assumption}

\newenvironment{proof}{\noindent {\bf Proof.}}{\hfill \hspace*{1pt}\hfill$\blacksquare$}

\begin{document}

\title{Geometry of deadbeat synchronization}
\author{S. Emre Tuna\footnote{The author is with Department of
Electrical and Electronics Engineering, Middle East Technical
University, 06800 Ankara, Turkey. Email: {\tt
tuna@eee.metu.edu.tr}}} \maketitle

\begin{abstract}
The deadbeat synchronization of identical discrete-time nonlinear
systems is studied from a geometric point of view. An array of
deadbeat observers coupled via a deadbeat interconnection is shown
to achieve synchronization in finite number of steps provided that
a compatibility condition is satisfied between the observer and
the interconnection. As an illustration to the theory, an example
is provided where an array of third order observers achieves
deadbeat synchronization.
\end{abstract}

\section{Introduction}

Two or more dynamical systems are said to synchronize when their
solutions converge to a common trajectory. The generality of this
definition allows many seemingly different cases to make examples
of synchronization \cite{wang02}. One such example is the
following pair of discrete-time linear systems
\begin{subeqnarray}\label{eqn:luenberger}
x_{1}^{+}&=&Ax_{1}\\
x_{2}^{+}&=&Ax_{2}+L(y_{1}-y_{2})
\end{subeqnarray}
where $y_{i}=Cx_{i}$ and all the eigenvalues of the matrix $A-LC$
are on the open unit disc. The second system
(\ref{eqn:luenberger}b) is the classic linear observer
\cite{rugh96} and its construction readily yields
$\|x_{1}(k)-x_{2}(k)\|\to 0$ as $k\to\infty$ regardless of the
initial conditions. Another linear example to synchronization is
the following array of systems with rather simple dynamics
\begin{eqnarray}\label{eqn:consensus}
y_{i}^{+}=\sum_{j=1}^{q}\gamma_{ij}y_{j}\,,\qquad
i=1,\,2,\,\ldots,\,q
\end{eqnarray}
where $[\gamma_{ij}]\in\Real^{q\times q}$ is a {\em (connected)
coupling matrix}. That is, $[\gamma_{ij}]$ satisfies: (i) the
entries of each row sum up to unity, which implies that
$\lambda=1$ is an eigenvalue with the eigenvector $[1\ \ 1\
\ldots\ 1]^{T}$, and (ii) all the remaining eigenvalues are on the
open unit disc. In this case the solutions $y_{i}(k)$ converge to
a fixed point in space and the systems are said to reach consensus
\cite{olfati07,moreau05}.

At first sight the
arrays~\eqref{eqn:luenberger}~and~\eqref{eqn:consensus} may not
seem to be relevant, but they are in fact the two limiting cases
of the following general structure
\begin{eqnarray}\label{eqn:tempting}
x_{i}^{+}&=&Ax_{i}+L\left(\sum_{j=1}^{q}\gamma_{ij}y_{j}-y_{i}\right)\,,\qquad
i=1,\,2,\,\ldots,\,q\\
y_{i}&=&Cx_{i}\,.\nonumber
\end{eqnarray}
Note that \eqref{eqn:tempting} boils down to
\eqref{eqn:luenberger} for
\begin{eqnarray*}
[\gamma_{ij}]=\left[\begin{array}{rr}1&0\\1&0\end{array}\right]
\end{eqnarray*}
and to \eqref{eqn:consensus} for $A=C=L=I$. An effective way to
study the synchronization behavior of an array is through
understanding the smaller pieces that it is made of
\cite{wu95,pecora98,tuna08}. So if the array~\eqref{eqn:tempting}
is what we are trying to understand then it is worthwhile to focus
on its two limiting cases: the array~\eqref{eqn:luenberger} and
the array~\eqref{eqn:consensus}. In this line of thinking, the
very first question that one is tempted to ask is the following.
{\em Given that the system~(\ref{eqn:luenberger}b) is an observer
for the system~(\ref{eqn:luenberger}a) and that the
array~\eqref{eqn:consensus} reaches consensus, does the
array~\eqref{eqn:tempting} synchronize?} However, a counterexample
is easy to construct and this naive guess has to be abandoned.

Having dispensed with the first question, we point our attention,
among a number of possibilities, to the following. {\em Given that
the system~(\ref{eqn:luenberger}b) is a deadbeat observer
\cite{glad83,valcher99,tuna12} for the
system~(\ref{eqn:luenberger}a) and that the
array~\eqref{eqn:consensus} reaches consensus in finite number of
steps \cite{kingston06,sundaram07}, does the
array~\eqref{eqn:tempting} synchronize?} This guess turns out to
be more fruitful than the first one. In fact not only
synchronization is achieved in this case, but it is achieved in
deadbeat fashion, that is, in finite number of steps. Motivated by
this simple observation on linear systems, we aim in this paper to
establish sufficient conditions that guarantee {\em deadbeat
synchronization} in an array of coupled identical discrete-time
nonlinear systems. What we particularly study here is the
synchronization behavior of an array of deadbeat observers that
are coupled through a fixed interconnection scheme, which itself,
if considered separately as the righthand side of an array, enjoys
deadbeat synchronization. We show that deadbeat synchronization is
achieved under a compatibility condition between the observer and
the interconnection.

The literature accommodates few results on deadbeat
synchronization. Motivated by possible applications in secure
communications, one of the earliest results on the subject is
presented in \cite{deangeli95}, where conditions for the
synchronization of {\em two systems} of type Lur'e, coupled via a
{\em scalar output} signal are given. Later, certain improvements
to this work are reported, for instance, in
\cite{lian02,grassi02,grassi12}, where synchronization still
requires that the number of systems is two and the output is
scalar. To the best of our knowledge, the problem of deadbeat
synchronization has not yet been considered in a general setting
where the number of (identical, nonlinear) individual systems are
arbitrary and the output signals, through which the systems are
coupled, are not necessarily scalar. The contribution of this
paper is hence intended to be in understanding better the
mechanism behind synchronization in discrete time from the
deadbeat point of view.

The remainder of the paper is organized as follows. The next
section contains some preliminary material. In
Section~\ref{sec:dbobs} is the construction of the deadbeat
observer, which is of geometric nature \cite{wonham70,wonham85}
and makes a special case of what is presented in \cite{tuna12}.
The reader will find an illustrative example following this
construction. As mentioned earlier, the observer together with the
system being observed make a particular case of synchronization
where there are only two systems. To reach a natural
generalization of this scenario we take two mental steps. First,
we remove the distinction between the two systems by allowing each
to observe the other. In other words, we dispense with the {\em
drive system (leader)-response system (follower)} hierarchy.
Second, having removed the distinction between the observer and
the observee, we allow the number of systems involved to be
arbitrary. At that point a method is required to couple this array
of observers. Therefore we introduce in Section~\ref{sec:dbint}
what we call {\em deadbeat interconnection}, which basically is a
nonlinear generalization of the time-invariant map (coupling
matrix) that appears in linear deadbeat consensus. Then in
Section~\ref{sec:dbsync} we bring together the observer
construction of Section~\ref{sec:dbobs} and the interconnection
scheme of Section~\ref{sec:dbint} to define the array of coupled
observers. There we establish the deadbeat synchronization of this
array under a compatibility condition that concerns both the
observer and the interconnection. In Section~\ref{sec:dbex} we
provide a nonlinear example where an array of third order deadbeat
observers are shown to achieve deadbeat synchronization.  Certain
issues are discussed in Section~\ref{sec:notes}.

\section{Preliminaries}

The set of nonnegative integers is denoted by $\Natural$, the set
of rational numbers by $\Rational$. A vector of all ones is
denoted by $\one$. The $m\times m$ identity matrix is denoted by
$I_{m}$, or sometimes simply by $I$ when what $m$ should be is
either obvious or immaterial. The symbol $\otimes$ denotes
Kronecker product. The null space of a matrix $M\in\Real^{m\times
n}$ is denoted by $\N(M)$ and $M^{\perp}$ denotes some real
matrix, whose columns form a basis for $\N(M)$. For square $M$ we
let $M^{0}=I$. Given a map $h:\X\to\Y$, $h^{-1}$ denotes the {\em
inverse} map in the general sense that, for $y\in\Y$, $h^{-1}(y)$
is the set of all $x\in\X$ satisfying $h(x)=y$. That is, we will
not need $h$ be bijective when talking about its inverse. For
$f:\X\to\X$ we let $f^{0}(x)=x$, $f^{k+1}(x)=f(f^{k}(x))$, and
$f^{-k}=(f^{-1})^{k}$ for $k\in\Natural$. Given vectors
$x_{1},\,x_{2},\,\ldots,\,x_{q}\in\Real^{n}$ we write
$(x_{1},\,x_{2},\,\ldots,\,x_{q})$ to mean $[x_{1}^{T}\ x_{2}^{T}\
\ldots\ x_{q}^{T}]^{T}\in\Real^{qn}$. We sometimes use ``$*$'' as
a placeholder for ``don't care.''

\section{Deadbeat observer}\label{sec:dbobs}

This section is dedicated to the description of the nonlinear
deadbeat observer. Later, in Section~\ref{sec:dbsync}, when we
establish the conditions for deadbeat synchronization of an array
of coupled observers, the construction presented here will be of
key importance. Unlike linear systems, there is not a standard
deadbeat observer construction for nonlinear systems. Even the
definition of a deadbeat observer may not be unique. The
definition and the construction that we present in this section
are adopted from \cite{tuna12}. The section ends with an
illustration of the construction.

\subsection{Definition}

Consider the following discrete-time system
\begin{eqnarray}\label{eqn:system}
x^{+}=f(x)\,,\quad y=h(x)
\end{eqnarray}
where $x\in\X\subset\Real^{n}$ is the {\em state}, $x^{+}$ is the
state at the next time instant, and $y\in
h(\X)=:\Y\subset\Real^{m}$ is the {\em output} or the {\em
measurement}. The {\em solution} of the system~\eqref{eqn:system}
at time $k\in\Natural$, starting at the initial condition
$x(0)\in\X$ is denoted by $x(k)$. Now consider the following array
\begin{subeqnarray}
x^{+}&=&f(x)\\
\xhat^{+}&=&g(\xhat,\,h(x))\,.\label{eqn:cascade}
\end{subeqnarray}
The solution of the system~(\ref{eqn:cascade}b) at time
$k\in\Natural$, starting at the initial condition $\xhat(0)\in\X$
is denoted by $\xhat(k)$. Note that $\xhat(k)$ depends also on
$x(0)$. We now use \eqref{eqn:cascade} to define deadbeat
observer.

\begin{definition}
Given $g:\X\times\Y\to\X$, the system
\begin{eqnarray*}
\xhat^{+}= g(\xhat,\,y)
\end{eqnarray*}
is said to be a {\em deadbeat observer for the
system~\eqref{eqn:system}} if there exists an integer $p\geq 1$
such that, for all initial conditions, the solutions of the
array~\eqref{eqn:cascade} satisfy $\xhat(k)=x(k)$ for all $k\geq
p$. The integer $p$ then is called a {\em deadbeat horizon}.
\end{definition}

\subsection{Construction}

To be used in the construction of the observer we define certain
sets associated with the system~\eqref{eqn:system}. For $x\in\X$
we let
\begin{eqnarray*}
[x]_{0}:=h^{-1}(h(x))\,.
\end{eqnarray*}
Note that when $h(x)=Cx$, where $C\in\Real^{m\times n}$, we have
$[x]_{0}=x+\N(C)$. We then let for $k\in\Natural$
\begin{eqnarray*}
[x]_{k+1}:=[x]_{k}^{+}\cap[x]_{0}
\end{eqnarray*}
where
\begin{eqnarray*}
[x]_{k}^{+}:=f([f^{-1}(x)]_{k})\,.
\end{eqnarray*}
We finally let
\begin{eqnarray*}
[x]_{-1}^{+}:=\X\,.
\end{eqnarray*}
We make the following two assumptions to guarantee that the
observer construction will work. We note that these conditions are
only sufficient. For less restrictive assumptions see
\cite{tuna12}.

\begin{assumption}\label{assume:bijective}
The map $f:\X\to\X$ is bijective.
\end{assumption}

\begin{assumption}\label{assume:singleton}
There exists $p\geq 1$ such that $[\xhat]_{p-2}^{+}\cap h^{-1}(y)$
is singleton for all $\xhat\in\X$ and $y\in\Y$.
\end{assumption}

The following result tells us how to design a deadbeat observer
under these assumptions.

\begin{theorem}\label{thm:dbobs}
Suppose Assumptions~\ref{assume:bijective}-\ref{assume:singleton}
hold. Then the system
\begin{eqnarray}\label{eqn:deadbeat}
\hat{x}^{+}= f([\xhat]^{+}_{p-2}\cap h^{-1}(y))
\end{eqnarray}
is a deadbeat observer for the system~\eqref{eqn:system} with
deadbeat horizon $p$.
\end{theorem}

Theorem~\ref{thm:dbobs} will later become a corollary of our main
result (Theorem~\ref{thm:main}). Hence we omit the proof. Let us
now give an example to this construction~\eqref{eqn:deadbeat}.

\subsection{Illustration}

Consider the system~\eqref{eqn:system} with
\begin{eqnarray*}
f(x)=\left[\begin{array}{c}1+\xi_{2}-a\xi_{1}^{2}\\b\xi_{1}+\xi_{3}\\-b\xi_{1}\end{array}
\right]
\end{eqnarray*}
where $x=(\xi_{1},\,\xi_{2},\,\xi_{3})\in\Real^{3}$. This map
appears in \cite{chua87} where it is reported to exhibit chaotic
behavior for certain values of real numbers $a$ and $b$. Let us
now construct a deadbeat observer taking $h(x)=\xi_{3}$. Note
first that $f$ is bijective (for $b$ nonzero) and its inverse is
\begin{eqnarray}\label{eqn:chuainv}
f^{-1}(x)=\left[\begin{array}{c}
-b^{-1}\xi_{3}\\
-1+ab^{-2}\xi_{3}^{2}+\xi_{1}\\
\xi_{2}+\xi_{3}
\end{array} \right]
\end{eqnarray}
Since $h(x)=\xi_{3}$ we can write
\begin{eqnarray}\label{eqn:chuax0}
[x]_{0}=\left\{\left[\begin{array}{c}
s\\
t\\
\xi_{3}
\end{array} \right]:s,\,t\in\Real\right\}
\end{eqnarray}
Then by~\eqref{eqn:chuainv}
\begin{eqnarray}\label{eqn:chuax0plus}
[x]_{0}^{+}
&=& f([f^{-1}(x)]_{0})\nonumber\\
&=& f\left(\left\{\left[\begin{array}{c}
s\\
t\\
\xi_{2}+\xi_{3}
\end{array} \right]:s,\,t\in\Real\right\}\right)\nonumber\\
%&=& \left\{f\left(\left[\begin{array}{c}
%s\\
%t\\
%\xi_{2}+\xi_{3}
%\end{array} \right]\right):s,\,t\in\Real\right\}\nonumber\\
&=& \left\{\left[\begin{array}{c}
t\\
s+\xi_{2}+\xi_{3}\\
-s
\end{array} \right]:s,\,t\in\Real\right\}
\end{eqnarray}
The set $[x]_{1}$ is defined to equal $[x]_{0}^{+}\cap[x]_{0}$.
The intersection of the sets~\eqref{eqn:chuax0} and
\eqref{eqn:chuax0plus} yields
\begin{eqnarray*}
[x]_{1}=\left\{\left[\begin{array}{c}
u\\
\xi_{2}\\
\xi_{3}
\end{array} \right]:u\in\Real\right\}
\end{eqnarray*}
Now we can construct $[x]_{1}^{+}$ as
\begin{eqnarray*}
[x]_{1}^{+}
&=& f([f^{-1}(x)]_{1})\nonumber\\
&=& f\left(\left\{\left[\begin{array}{c}
u\\
-1+ab^{-2}\xi_{3}^{2}+\xi_{1}\\
\xi_{2}+\xi_{3}
\end{array} \right]:u\in\Real\right\}\right)\nonumber\\
&=& \left\{\left[\begin{array}{c}
\xi_{1}+ab^{-2}\xi_{3}^{2}-au^{2}\\
bu+\xi_{2}+\xi_{3}\\
-bu
\end{array} \right]:u\in\Real\right\}
\end{eqnarray*}
Observe that
\begin{eqnarray*}
[\xhat]_{1}^{+}\cap h^{-1}(y) &=&\left\{\left[\begin{array}{c}
\hat\xi_{1}+ab^{-2}\hat\xi_{3}^{2}-au^{2}\\
bu+\hat\xi_{2}+\hat\xi_{3}\\
-bu
\end{array} \right]:u\in\Real\right\}
\cap\left\{\left[\begin{array}{c}
s\\
t\\
y
\end{array} \right]:s,\,t\in\Real\right\}\\
&=&\left[\begin{array}{c}
\hat\xi_{1}+ab^{-2}(\hat\xi_{3}^{2}-y^{2})\\
\hat\xi_{2}+\hat\xi_{3}-y\\
y
\end{array}\right]
\end{eqnarray*}
which means that Assumption~\ref{assume:singleton} is satisfied
with $p=3$. The dynamics of the deadbeat observer then read
\begin{eqnarray*}
\xhat^{+}&=&f([\xhat]_{1}^{+}\cap h^{-1}(y))\\
&=&\left[\!\!\begin{array}{c}
1+\hat\xi_{2}+\hat\xi_{3}-y-a(\hat\xi_{1}+ab^{-2}(\hat\xi_{3}^{2}-y^{2}))^{2}\\
b\hat\xi_{1}+ab^{-1}(\hat\xi_{3}^{2}-y^{2})+y\\
-b\hat\xi_{1}-ab^{-1}(\hat\xi_{3}^{2}-y^{2})
\end{array}\!\!\right]
\end{eqnarray*}
For the parameter choice $a=1$ and $b=1/3$, Fig.~\ref{fig:dbobs}
shows the simulation results for the initial conditions
$x(0)=(1,\,-1,\,1)$ and $\xhat(0)=(0,\,2,\,1)$.
\begin{figure}[h]
\begin{center}
\includegraphics[scale=0.6]{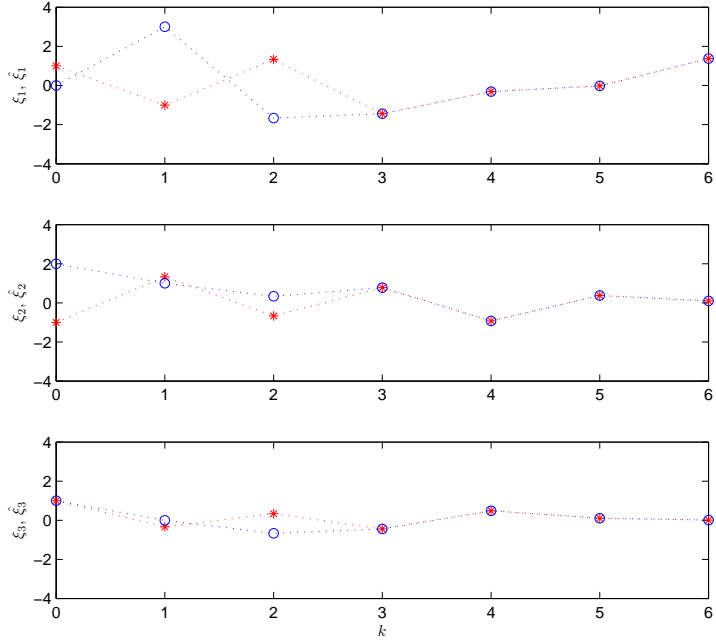}\caption{Evolution of states $x$ (marked $*$) and $\xhat$ (marked
$\circ$).}\label{fig:dbobs}
\end{center}
\end{figure}

The description of the deadbeat observer is now complete. As
mentioned earlier, we are headed towards understanding the
collective behavior of an arbitrary number of identical observers
that are interacting. To be able to proceed, we therefore first
need to be precise with what we mean by interacting. To this end,
we introduce in the next section the so called deadbeat
interconnection. This interconnection scheme will be evoked in
Section~\ref{sec:dbsync} to characterize the coupling of the array
whose synchronization we will study.

\section{Deadbeat interconnection}\label{sec:dbint}

Here we provide a generalization of the case where the linear
array~\eqref{eqn:consensus} reaches consensus in finite number of
steps, which happens when the characteristic polynomial of the
$q\times q$ coupling matrix $[\gamma_{ij}]$ is
$d(s)=s^{q-1}(s-1)$. When referring to such $[\gamma_{ij}]$ we
will use the term {\em deadbeat coupling matrix}. A primitive
example for $q=4$ is given as
\begin{eqnarray}\label{eqn:dbcoup}
[\gamma_{ij}]=\left[\begin{array}{rrrr}
0.4&-0.2&3.2&-2.4\\
0.4&-0.2&0.2&0.6\\
0.2&-0.6&0.6&0.8\\
0.3&-0.4&0.9&0.2
\end{array}\right]
\end{eqnarray}
which yields
\begin{eqnarray*}
[\gamma_{ij}]^{k}=\left[\begin{array}{rrrr}
0.1&-0.8&1.3&0.4\\
0.1&-0.8&1.3&0.4\\
0.1&-0.8&1.3&0.4\\
0.1&-0.8&1.3&0.4
\end{array}\right]=\left[\begin{array}{r}
1\\1\\1\\1
\end{array}\right]\times[0.1 \ \ -0.8\ \ 1.3\ \ 0.4]
\end{eqnarray*}
for all $k\geq 3$. This means that the solutions of the
array~\eqref{eqn:consensus} should satisfy
$y_{i}(k)=0.1y_{1}(0)-0.8y_{2}(0)+1.3y_{3}(0)+0.4y_{4}(0)$ for all
$k\geq 3$. That is, convergence is exact. Now we give the
generalization.

\begin{definition}\label{def:dbint}
A map $\gamma=(\gamma_{1},\,\gamma_{2},\,\ldots,\,\gamma_{q})$
with $\gamma_{i}:\Y^{q}\to\Y$ is said to be a {\em deadbeat
interconnection} if the following conditions simultaneously hold.
\begin{itemize}
\item There exists an integer $r\geq 1$ such that, for all initial
conditions, the solutions of the array
\begin{eqnarray*}
y_{i}^{+}&=&\gamma_{i}(y_{1},\,y_{2},\,\ldots,\,y_{q})\,,\qquad
i=1,\,2,\,\ldots,\,q
\end{eqnarray*}
satisfy $y_{i}(k)=y_{j}(k)$ for all $k\geq r$ and all $i,\,j$. The
integer $r$ then is called a {\em deadbeat horizon}. \item For all
$y\in\Y$ we have
$\gamma_{i}(y,\,y,\,\ldots,\,y)=\gamma_{j}(y,\,y,\,\ldots,\,y)$
for all $i,\,j$.
\end{itemize}
\end{definition}

Some examples are in order. Let $G\in\Real^{q\times q}$ be a
deadbeat coupling matrix and $\Y=\Real^{m}$. Then the linear map
\begin{eqnarray}\label{eqn:struc1}
\vay\mapsto(G\otimes Q)\vay
\end{eqnarray}
with $\vay=(y_{1},\,y_{2},\,\ldots,\, y_{q})$ is a deadbeat
interconnection for any $Q\in\Real^{m\times m}$. Note for $Q=I$
that the array~\eqref{eqn:consensus} makes a special case of this
construction. Not all linear deadbeat interconnections must have
this structure~\eqref{eqn:struc1} though. For instance, for
$\Y=\Real^{2}$ and $q=2$, the map $\vay\mapsto\Gamma\vay$ with
\begin{eqnarray}\label{eqn:struc2}
\Gamma=\left[\begin{array}{rrrr}0&0&1&0\\0&1&0&0\\0&0&1&0\\
1&1&-1&0\end{array}\right]
\end{eqnarray}
can be shown to be a deadbeat interconnection for which no
$G,\,Q\in\Real^{2\times 2}$ exist that yield $G\otimes Q=\Gamma$.
Our last example is nonlinear. Let $[\gamma_{ij}]\in\Real^{q\times
q}$ be a deadbeat coupling matrix and $\Y=\Real$. Then the map
$\gamma=(\gamma_{1},\,\gamma_{2},\,\ldots,\,\gamma_{q})$ with
\begin{eqnarray}\label{eqn:hmgcoup}
\gamma_{i}(\vay)=\left(\sum_{j=1}^{q}\gamma_{ij}y_{j}^3\right)^{1/3}
\end{eqnarray}
is a deadbeat interconnection.

Recall that in the standard observer setting, where there are only
two systems in the array, i.e., the response system (the observer)
and the drive system (the system being observed); the driving
signal for the observer is the output of the system being
observed. When one considers the general setting, where the array
contains an arbitrary number of systems, the driving signal of
each system (observer) is a function of the outputs of all the
systems in the array. An array with $q$ many systems means there
will be $q$ such functions. The deadbeat interconnection
$(\gamma_{1},\,\gamma_{2},\,\ldots,\,\gamma_{q})$ we defined in
this section is nothing but a particular collection of those
coupling functions. Having defined deadbeat observer and deadbeat
interconnection, now we proceed to study the array that we form by
bringing the two together.

\section{Deadbeat synchronization}\label{sec:dbsync}

This section is where we finally gather the conditions that yield
deadbeat synchronization of an array of coupled identical
observers. Let us begin with defining the phenomenon under
investigation.

\begin{definition}
Given the maps $g:\X\times\Y\to\X$, $h:\X\to\Y$, and
$(\gamma_{1},\,\gamma_{2},\,\ldots,\,\gamma_{q})=\gamma:\Y^{q}\to\Y^{q}$,
the following array
\begin{eqnarray}\label{eqn:arraygeneral}
x_{i}^{+}&=&g(x_{i},\,w_{i})\,,\qquad i=1,\,2,\,\ldots,\,q\\
w_{i}&=&\gamma_{i}(h(x_{1}),\,h(x_{2}),\,\ldots,\,h(x_{q}))\nonumber
\end{eqnarray}
is said to achieve {\em deadbeat synchronization} if there exists
an integer $\tau\geq 1$ such that, for all initial conditions, the
solutions satisfy $x_{i}(k)=x_{j}(k)$ for all $k\geq\tau$ and all
$i,\,j$. The integer $\tau$ then is called a {\em deadbeat
horizon}.
\end{definition}

This definition lets us write the formal statement of the problem
to which we propose a solution in this paper: {\em Under what
conditions on the triple $(g,\,h,\,\gamma)$ does the
array~\eqref{eqn:arraygeneral} achieve deadbeat synchronization?}
Now, instead of directly listing the assumptions and establishing
the main result, we prefer first to present the linear case, which
we mentioned in the introduction in a slightly less general
form~\eqref{eqn:tempting}, that motivated all the analysis in this
paper. The conditions on which this linear result is founded we
take as a justification for some of the assumptions we will have
made.

\begin{theorem}\label{thm:linsync}
Let $A\in\Real^{n\times n}$, $C\in\Real^{m\times n}$, and
$L\in\Real^{n\times m}$ be such that $A-LC$ is nilpotent. Let
$[\gamma_{ij}]\in\Real^{q\times q}$ be a deadbeat coupling matrix.
Then the array
\begin{eqnarray}\label{eqn:moretempting}
x_{i}^{+}&=&Ax_{i}+L\left(\sum_{j=1}^{q}\gamma_{ij}Qy_{j}-y_{i}\right)\,,\qquad
i=1,\,2,\,\ldots,\,q\\
y_{i}&=&Cx_{i}\nonumber
\end{eqnarray}
achieves deadbeat synchronization for all $Q\in\Real^{m\times m}$.
In particular, for all initial conditions, we have
$x_{i}(k)=x_{j}(k)$ for all $k\geq rp$ and all $i,\,j$; where the
integers $r\geq 1$ and $p\geq 1$ satisfy
$[\gamma_{ij}]^{r}=[\gamma_{ij}]^{r+1}$ and $(A-LC)^{p}=0$.
\end{theorem}

\begin{proof}
First let $\ex=(x_{1},\,x_{2},\,\ldots,\,x_{q})$ and
$G:=[\gamma_{ij}]$. Then the array~\eqref{eqn:moretempting} can be
written as
\begin{eqnarray*}
\ex^{+}=\left[I_{q}\otimes A+(I_{q}\otimes L)(G\otimes
Q-I_{qm})(I_{q}\otimes C)\right]\ex\,.
\end{eqnarray*}
Since $G$ is a deadbeat coupling matrix we have
$G^{r}=\one\ell^{T}$ where $\ell\in\Real^{q}$ is the left
eigenvector of $G$ for the eigenvalue $\lambda=1$ satisfying
$\ell^{T}\one=1$. That all the remaining eigenvalues are at the
origin allows us to find a transformation matrix
$V\in\Real^{q\times q}$ satisfying
\begin{eqnarray*}
V=[\one\ *]\quad\mbox{and}\quad V^{-1}=[\ell\ *]^{T}
\end{eqnarray*}
as well as
\begin{eqnarray*}
V^{-1}GV=\left[\begin{array}{cc} 1&0\\0&J
\end{array}\right]
\end{eqnarray*}
for some $J\in\Real^{(q-1)\times(q-1)}$ that is strictly upper
triangular. Without loss of generality we assume that $J$ is in
Jordan form. That is, $J = {\rm
diag}(J_{1},\,J_{2},\,\ldots,\,J_{\sigma})$ where each (Jordan)
block $J_{\alpha}$, $\alpha\in\{1,\,2,\,\ldots,\,\sigma\}$, is
strictly upper triangular. Then the size of the largest Jordan
block is no greater than $r\times r$. Employing the coordinate
change $\zi:=(V^{-1}\otimes I_{n})\ex$ we can write
\begin{eqnarray}\label{eqn:zsystem}
\zi^{+}&=&(V^{-1}\otimes I_{n})\left[I_{q}\otimes A+(I_{q}\otimes
L)(G\otimes Q-I_{qm})(I_{q}\otimes C)\right](V\otimes I_{n})\zi\nonumber\\
&=&\left[I_{q}\otimes A+(I_{q}\otimes L)(V^{-1}\otimes
I_{m})(G\otimes Q-I_{qm})(V\otimes I_{m})(I_{q}\otimes
C)\right]\zi\nonumber\\
&=&\left[I_{q}\otimes A+(I_{q}\otimes L)(V^{-1}GV\otimes
Q-I_{qm})(I_{q}\otimes C)\right]\zi\nonumber\\
&=&\left[\begin{array}{cc}A+L(Q-I_{m})C&0\\0&N\end{array}\right]\zi
\end{eqnarray}
where
\begin{eqnarray*}
N=I_{q-1}\otimes A+(I_{q-1}\otimes L)(J\otimes
Q-I_{(q-1)m})(I_{q-1}\otimes C)\,.
\end{eqnarray*}
Since $J$ is block diagonal with blocks strictly upper triangular,
$N$ is also block diagonal $N = {\rm
diag}(N_{1},\,N_{2},\,\ldots,\,N_{\sigma})$ with the structure
\begin{eqnarray*}
N_{\alpha}=\left[\begin{array}{ccccc}A-LC&*&*&\cdots&*\\0&A-LC&*&\cdots&*\\0&0&A-LC&\cdots&*\\\vdots&\vdots&\vdots&\ddots&\vdots\\0&0&0&\cdots&A-LC\end{array}\right]
\end{eqnarray*}
for all $\alpha\in\{1,\,2,\,\ldots,\,\sigma\}$. Note that the size
of the largest of the blocks $N_{\alpha}$ is no greater than
$rn\times rn$. And this largest block must vanish in $rp$ steps
because $(A-LC)^{p}=0$. By the time the largest block has vanished
the other blocks must have also vanished. We deduce therefore
\begin{eqnarray}\label{eqn:Npr}
N^{rp}=0\,.
\end{eqnarray}
Equations \eqref{eqn:zsystem} and \eqref{eqn:Npr} yield
\begin{eqnarray*}
\ex(k)=[\one\ell^{T}\otimes(A+L(Q-I_{m})C)^{k}]\ex(0)
\end{eqnarray*}
for all $k\geq rp$. In other words,  all the solutions
$x_{i}(\cdot)$ converge (in deadbeat fashion) to the trajectory
\begin{eqnarray*}
{\bar x}(k)=(A+LQC-LC)^{k}\bar{x}(0)
\end{eqnarray*}
where $\bar x(0)=(\ell^{T}\otimes I_{n})\ex(0)$.
\end{proof}

\begin{remark}
The proof of Theorem~\ref{thm:linsync} reveals that the
array~\eqref{eqn:moretempting} would still achieve synchronization
(not in finite number of steps, but asymptotically) with deadbeat
interconnection even if the matrix $A-LC$ was not nilpotent, but
only convergent, i.e., $(A-LC)^k\to 0$ as $k\to\infty$.
\end{remark}

\begin{remark}\label{rem:timeinvariance}
Another point made apparent by the proof of
Theorem~\ref{thm:linsync} is the following. Instead of the fixed
interconnection $\vay\mapsto(G\otimes Q)\vay$ if we employed a
time-varying one with the structure $\vay\mapsto(G\otimes
Q_{k})\vay$, the array~\eqref{eqn:moretempting} would still
achieve deadbeat synchronization (with the same deadbeat horizon)
regardless of how matrix $Q_{k}$ changes in time.
\end{remark}

Let us go through the conditions required for
Theorem~\ref{thm:linsync} in order to generate their nonlinear
counterparts. One condition is that the matrix $A-LC$ is
nilpotent. This is equivalent to that the
system~(\ref{eqn:luenberger}b) is a deadbeat observer for the
system~(\ref{eqn:luenberger}a), which suggests for the nonlinear
case that each individual system of the array is a deadbeat
observer. Recall that, under
Assumptions~\ref{assume:bijective}-\ref{assume:singleton}, we know
how to construct a deadbeat observer, see Theorem~\ref{thm:dbobs}.
Another condition in Theorem~\ref{thm:linsync} is that
$[\gamma_{ij}]$ is a deadbeat coupling matrix. This translates
into that the map $\vay\mapsto([\gamma_{ij}]\otimes Q)\vay$ is a
deadbeat interconnection. Hence the following assumption, which we
will later need for the main theorem.

\begin{assumption}\label{assume:interconn}
The map $\gamma=(\gamma_{1},\,\gamma_{2},\,\ldots,\,\gamma_{q})$
is a deadbeat interconnection with deadbeat horizon $r$.
\end{assumption}

Consider now, in the light of Theorem~\ref{thm:dbobs} and under
Assumptions~\ref{assume:bijective}-\ref{assume:interconn}, the
following array
\begin{eqnarray}\label{eqn:dbsync}
x_{i}^{+}&=&f([x_{i}]_{p-2}^{+}\cap h^{-1}(w_{i}))\,,\qquad
i=1,\,2,\,\ldots,\,q\\
w_{i}&=&\gamma_{i}(h(x_{1}),\,h(x_{2}),\,\ldots,\,h(x_{q}))\,.\nonumber
\end{eqnarray}
Does the array~\eqref{eqn:dbsync} achieve deadbeat
synchronization? The answer to this question is negative and that
is why we will eventually need to make a fourth assumption. A
counterexample, which is indeed linear, is as follows.

We take $\X=\Real^{4}$ and $\Y=\Real^{2}$. Let $f(x)=Ax$ and
$h(x)=Cx$ with
\begin{eqnarray*}
A=\left[\begin{array}{rrrr}0&-1&1&0\\-1&0&0&-1\\1&0&0&0\\
-1&-1&-1&1\end{array}\right]\quad\mbox{and}\quad
C=\left[\begin{array}{rrrr}0&-1&1&-1\\-1&1&0&1\end{array}\right]
\end{eqnarray*}
Assumption~\ref{assume:bijective} is satisfied since $A$ is
nonsingular. Assumption~\ref{assume:singleton} is also satisfied
(with $p=2$) because
\begin{eqnarray*}
[x]_{0}^{+}\cap h^{-1}(y)=x+H(y-Cx)
\end{eqnarray*}
where $H$ can be computed as
\begin{eqnarray*}
H = AC^{\perp}(CAC^{\perp})^{-1}=\left[\begin{array}{rr}-1&-5/4\\1&6/4\\0&-1/4\\
-2&-7/4\end{array}\right]
\end{eqnarray*}
Note that $(A-AHC)^2=0$. As for the deadbeat interconnection, we
take $\gamma(\vay)=\Gamma\vay$ where $\Gamma\in\Real^{4\times 4}$
is as in \eqref{eqn:struc2}, which is known to satisfy
Assumption~\ref{assume:interconn} without admitting matrices
$G,\,Q\in\Real^{2\times 2}$ to realize $G\otimes Q=\Gamma$. (We
especially want to emphasize here that the special
structure~\eqref{eqn:struc1} of the interconnection assumed in
Theorem~\ref{thm:linsync} is not merely for demonstrational
convenience. The structure~\eqref{eqn:struc1} does indeed play a
role in achieving synchronization.) Hence the triple
$(f,\,h,\,\gamma)=(A,\,C,\,\Gamma)$ satisfies
Assumptions~\ref{assume:bijective}-\ref{assume:interconn}. Under
our parameter choice the array~\eqref{eqn:dbsync} enjoys the
following form
\begin{eqnarray*}
\left[\begin{array}{c}x_{1}\\
x_{2}\end{array}\right]^{+}=\underbrace{(I_{2}\otimes A)(I_{8}
+(I_{2}\otimes H)(\Gamma-I_{4})(I_{2}\otimes C))}_{\displaystyle\Phi}\left[\begin{array}{c}x_{1}\\
x_{2}\end{array}\right]
\end{eqnarray*}
which does not achieve deadbeat synchronization. Because if it did
then it would necessarily require that the matrix
$\Phi\in\Real^{8\times 8}$ has at least $(q-1)n=4$ of its
eigenvalues at the origin. However, the characteristic polynomial
of $\Phi$ turns out to be
$d(s)=s^8-3.5s^7-1.5s^6+11.5s^5-2.5s^4-8s^3-2s^2$.

This example justifies that
Assumptions~\ref{assume:bijective}-\ref{assume:interconn} are not
sufficient and additional conditions are needed for the deadbeat
synchronization of the array~\eqref{eqn:dbsync}. Let us now
provide (in Definition~\ref{def:compatible}) one such condition.
First, however, we need to introduce some notation associated with
the triple $(f,\,h,\,\gamma)$.

Recall $\vay=(y_{1},\,y_{2},\,\ldots,\, y_{q})$ and
${\ex}=(x_{1},\,x_{2},\,\ldots,\, x_{q})$. By some abuse of
notation we let
$f({\ex})=(f(x_{1}),\,f(x_{2}),\,\ldots,\,f(x_{q}))$. For
instance, if $f(x)=Ax$, then $f(\ex)=(I_{q}\otimes A)\ex$.
Likewise, $h({\ex})=(h(x_{1}),\,h(x_{2}),\,\ldots,\,h(x_{q}))$.
Also, for $\sigma\in\Natural$, we define
\begin{eqnarray*}
\Y_{\sigma}:=\{\vay\in\Y^{q}:\gamma^{\sigma}(\vay)\in\Y_{0}\}
\end{eqnarray*}
where
\begin{eqnarray*}
\Y_{0}:=\{\vay\in\Y^{q}:y_{1}=y_{2}=\cdots=y_{q}\}\,.
\end{eqnarray*}

\begin{remark}
The notation introduced here allows us to rephrase
Definition~\ref{def:dbint} more compactly as follows. A deadbeat
interconnection $\gamma:\Y^{q}\to\Y^{q}$ is such that
$\gamma(\Y_{0})\subset\Y_{0}$ and $\gamma^r(\Y^{q})\subset\Y_{0}$
for some $r\geq 1$.
\end{remark}

\begin{definition}\label{def:compatible}
The triple $(f,\,h,\,\gamma)$ is said to be {\em compatible} if,
for all $\sigma\geq 1$ and $\ex\in\X^{q}$,
\begin{eqnarray*}
h(f^{k}(\ex))\in\Y_{\sigma}\quad \forall
k\in\{0,\,1,\,\ldots,\,p-1\}\implies h(f^{p}(\ex))\in\Y_{\sigma}
\end{eqnarray*}
where $p$ is as in Assumption~\ref{assume:singleton}.
\end{definition}

Here is our last assumption.

\begin{assumption}\label{assume:compatibl}
The triple $(f,\,h,\,\gamma)$ is compatible.
\end{assumption}

We need some sort of justification for this assumption. When
studying nonlinear systems, a first step towards forming an
opinion on whether an assumption is too restrictive or not for the
goal to be achieved is to see what it boils down to for linear
systems. For this purpose we want to point out that, for the
linear array~\eqref{eqn:moretempting}, compatibility is implied by
Assumptions~\ref{assume:bijective}-\ref{assume:interconn} whenever
the matrix $Q$ is nonsingular. The next theorem formalizes this.

\begin{theorem}\label{thm:linear}
Let $A\in\Real^{n\times n}$ be nonsingular and $C\in\Real^{m\times
n}$ be such that there exists $L\in\Real^{n\times m}$ satisfying
$(A-LC)^{p}=0$. Then, given $Q\in\Real^{m\times m}$ nonsingular
and a deadbeat coupling matrix $G\in\Real^{q\times q}$, the triple
$(A,\,C,\,G\otimes Q)$ is compatible.
\end{theorem}

\begin{proof}
Since $G$ is a deadbeat coupling matrix, for some integer $r\leq
q-1$, we have $G^{r}=G^{r+1}=\one\ell^{T}$ where
$\ell\in\Real^{q}$ is the left eigenvector of $G$ for the
eigenvalue $\lambda=1$ satisfying $\ell^{T}\one=1$. That $Q$ is
nonsingular allows us to write
$\Y_{\sigma}=\N((G^{\sigma}-\one\ell^{T})\otimes I_{m})$. Given
$\sigma$, let $\{v_{1},\,v_{2},\,\ldots,\,v_{\alpha}\}$ be a basis
for $\N(G^{\sigma}-\one\ell^{T})$ and
$\{e_{1},\,e_{2},\,\ldots,\,e_{m}\}$ be a basis for $\Real^{m}$.
Then the set $\bigcup_{i,\,j}\{v_{i}\otimes e_{j}\}$ makes a basis
for $\Y_{\sigma}$. Now, given $\ex\in\Real^{qn}$, suppose
$(I_{q}\otimes CA^{k})\ex\in\Y_{\sigma}$ for all
$k=0,\,1,\,\ldots,\,p-1$. Note that we can find scalars
$\alpha_{kij}$ such that
\begin{eqnarray*}
(I_{q}\otimes
CA^{k})\ex=\sum_{i=1}^{\alpha}\sum_{j=1}^{m}\alpha_{kij}(v_{i}\otimes
e_{j})\,.
\end{eqnarray*}
Since $A$ is nonsingular, that we can find some $L$ satisfying
$(A-LC)^p=0$ implies $[C^{T}\ \ A^{T}C^{T}\ \ \ldots$
$A^{(p-1)T}C^{T}]$ is full row rank. Hence, for each
$k=0,\,1,\,\ldots,\,p-1$, we can find $M_{k}\in\Real^{m\times m}$
such that
\begin{eqnarray*}
\sum_{k=0}^{p-1}M_{k}CA^{k}=CA^{p}\,.
\end{eqnarray*}
We can then write
\begin{eqnarray*}
(I_{q}\otimes CA^{p})\ex &=&\sum_{k=0}^{p-1}(I_{q}\otimes
M_{k}CA^{k})\ex\\
&=&\sum_{k=0}^{p-1}(I_{q}\otimes M_{k})(I_{q}\otimes
CA^{k})\ex\\
&=&\sum_{k=0}^{p-1}(I_{q}\otimes
M_{k})\sum_{i=1}^{\alpha}\sum_{j=1}^{m}\alpha_{kij}(v_{i}\otimes e_{j})\\
&=&\sum_{k=0}^{p-1}\sum_{i=1}^{\alpha}\sum_{j=1}^{m}\alpha_{kij}(v_{i}\otimes
M_{k}e_{j})\\
&\in&\Y_{\sigma}
\end{eqnarray*}
because $((G^{\sigma}-\one\ell^{T})\otimes I_{m})(v_{i}\otimes
M_{k}e_{j})=(G^{\sigma}-\one\ell^{T})v_{i}\otimes M_{k}e_{j}=0$.
\end{proof}
\medskip\smallskip

The next theorem is our main result, which says that if a number
of identical deadbeat observers are coupled through a deadbeat
interconnection, the array achieves synchronization in finite
number of steps provided that the observer and the interconnection
satisfy the compatibility condition of
Definition~\ref{def:compatible}.

\begin{theorem}\label{thm:main}
Under Assumptions~\ref{assume:bijective}-\ref{assume:compatibl},
the array~\eqref{eqn:dbsync} achieves deadbeat synchronization
with deadbeat horizon $\tau=rp$.
\end{theorem}

The demonstration of Theorem~\ref{thm:main} requires some
preliminary results, which we provide in the sequel as three
lemmas. (The lemmas will implicitly posit
Assumptions~\ref{assume:bijective}-\ref{assume:interconn}.) The
first of these results concerns some properties of the set-valued
functions $x\mapsto[x]_{k}$ and $x\mapsto[x]_{k}^{+}$.
\medskip\smallskip

\noindent{\bf Caveat.} Henceforth, confined to this section only,
we will avoid the standard use of parentheses when the risk of
confusion is negligible. For instance, $h^{-1}(h(x))$ will be
replaced by $h^{-1}hx$.

\begin{lemma}\label{lem:list}
Let $x\in\X$ and $k,\,\alpha\in\Natural$. The following hold.
\begin{enumerate}
\item $x\in[x]_{k}$ and $x\in[x]_{k}^{+}$.

\item $[[x]_{k+\alpha}]_{k}=[x]_{k}$ and
$[[x]_{k+\alpha}^{+}]_{k}^{+}=[x]_{k}^{+}$.
\end{enumerate}
\end{lemma}

\begin{proof}
We begin with proving the first property. Notice that $x \in
h^{-1}hx=[x]_{0}$. Now suppose $x\in[x]_{k}$ for some $k$. Then we
can write (since $f$ is bijective) $x=ff^{-1}x\in
f[f^{-1}x]_{k}=[x]_{k}^{+}$. Thence $x
\in[x]_{k}^{+}\cap[x]_{0}=[x]_{k+1}$. The result follows by
induction.

Now we demonstrate the second property. Observe that if
$z\in[x]_{0}$ then $hz=hx$, which implies $[z]_{0}=[x]_{0}$. We
claim for all $k$ that
\begin{eqnarray}\label{eqn:yemek}
z\in[x]_{k} \implies [z]_{k}=[x]_{k}\,.
\end{eqnarray}
To prove our claim we suppose that it holds for some $k$. Let
$x,\,z$ be such that $z\in[x]_{k+1}=f[f^{-1}x]_{k}\cap[x]_{0}$.
This means that $z\in[x]_{0}$, which yields $[z]_{0}=[x]_{0}$.
Also,
\begin{eqnarray*}
z\in f[f^{-1}x]_{k}&\implies&f^{-1}z\in [f^{-1}x]_{k}\\
&\implies&[f^{-1}z]_{k} = [f^{-1}x]_{k}\,.
\end{eqnarray*}
Then we can write
\begin{eqnarray*}
[z]_{k+1}&=&f[f^{-1}z]_{k}\cap[z]_{0}\\
&=&f[f^{-1}x]_{k}\cap[x]_{0}\\
&=&[x]_{k+1}
\end{eqnarray*}
whence \eqref{eqn:yemek} follows by induction. Notice that
$[x]_{1}\subset[x]_{0}$ since $[x]_{1}=f[f^{-1}x]_{0}\cap[x]_{0}$.
Our second claim is
\begin{eqnarray}\label{eqn:tursu}
[x]_{k+\alpha}\subset[x]_{k}
\end{eqnarray}
for all $k$ and $\alpha$. Note that it is enough to establish this
for $\alpha=1$. Again we employ induction. Suppose
\eqref{eqn:tursu} holds with $\alpha=1$ for some $k$. Then
\begin{eqnarray*}
[x]_{k+2}
&=&f[f^{-1}x]_{k+1}\cap[x]_{0}\\
&\subset&f[f^{-1}x]_{k}\cap[x]_{0}\\
&=&[x]_{k+1}\,.
\end{eqnarray*}
Now, by \eqref{eqn:yemek} and \eqref{eqn:tursu} we have
\begin{eqnarray*}
[[x]_{k+\alpha}]_{k}&=&\bigcup_{z\in[x]_{k+\alpha}}[z]_{k}\\
&\subset&\bigcup_{z\in[x]_{k}}[z]_{k}\\
&=&[x]_{k}\,.
\end{eqnarray*}
By the first property we can write
$[x]_{k}\subset[[x]_{k+\alpha}]_{k}$. Hence
$[x]_{k}=[[x]_{k+\alpha}]_{k}$. Moreover,
\begin{eqnarray*}
[[x]_{k+\alpha}^{+}]_{k}^{+}
&=& f[f^{-1}f[f^{-1}x]_{k+\alpha}]_{k}\\
&=& f[[f^{-1}x]_{k+\alpha}]_{k}\\
&=& f[f^{-1}x]_{k}\\
&=& [x]_{k}^{+}
\end{eqnarray*}
which completes the demonstration of Lemma~\ref{lem:list}.
\end{proof}
\medskip\smallskip

Note that the array~\eqref{eqn:dbsync} leads to the following
system in $\X^{q}$
\begin{eqnarray}\label{eqn:formidable}
\ex^{+}=f([\ex]_{p-2}^{+}\cap h^{-1}\gamma h\ex)
\end{eqnarray}
where
$[\ex]_{p-2}^{+}=([x_{1}]_{p-2}^{+},\,[x_{2}]_{p-2}^{+},\,\ldots,\,[x_{q}]_{p-2}^{+})$.
The next two results concern the system~\eqref{eqn:formidable}.

\begin{lemma}\label{lem:one}
Consider the system~\eqref{eqn:formidable}. For all $k\geq p-1$
the solution $\ex(k)=:\ex_{k}$ satisfies
\begin{eqnarray}\label{eqn:ayierdal}
\ex_{k}\in f^{\alpha-1}[\ex_{k-\alpha+1}]_{p-\alpha}^{+}
\end{eqnarray}
for all $\alpha\in\{1,\,2,\,\ldots,\,p\}$.
\end{lemma}

\begin{proof}
Suppose \eqref{eqn:ayierdal} holds for some
$\alpha\in\{1,\,2,\,\ldots,\,p-1\}$. Then we can write by
\eqref{eqn:formidable} and Lemma~\ref{lem:list}
\begin{eqnarray*}
\ex_{k}
&\in& f^{\alpha-1}[\ex_{k-\alpha+1}]_{p-\alpha}^{+}\\
&=& f^{\alpha-1}f[f^{-1}\ex_{k-\alpha+1}]_{p-\alpha}\\
&=& f^{\alpha}[f^{-1}\ex_{k-\alpha+1}]_{p-\alpha}\\
&=& f^{\alpha}[f^{-1}f([\ex_{k-\alpha}]_{p-2}^{+}\cap h^{-1}\gamma h\ex_{k-\alpha})]_{p-\alpha}\\
&=& f^{\alpha}[[\ex_{k-\alpha}]_{p-2}^{+}\cap h^{-1}\gamma h\ex_{k-\alpha}]_{p-\alpha}\\
&\subset& f^{\alpha}[[\ex_{k-\alpha}]_{p-2}^{+}]_{p-\alpha}\\
&=& f^{\alpha}([[\ex_{k-\alpha}]_{p-2}^{+}]_{p-\alpha-1}^{+}\cap[[\ex_{k-\alpha}]_{p-2}^{+}]_{0})\\
&\subset& f^{\alpha}[[\ex_{k-\alpha}]_{p-2}^{+}]_{p-\alpha-1}^{+}\\
&=& f^{\alpha}[\ex_{k-\alpha}]_{p-\alpha-1}^{+}\,.
\end{eqnarray*}
By Lemma~\ref{lem:list} we also have
$\ex_{k}\in[\ex_{k}]_{p-1}^{+}$ which verifies
\eqref{eqn:ayierdal} for $\alpha=1$. The result then follows by
induction.
\end{proof}

\begin{lemma}\label{lem:two}
Consider the system~\eqref{eqn:formidable}. For all $k\geq p$ the
solution satisfies
\begin{eqnarray*}
hf^{-\alpha}\ex_{k}=\gamma h\ex_{k-\alpha}
\end{eqnarray*}
for all $\alpha\in\{1,\,2,\,\ldots,\,p\}$.
\end{lemma}

\begin{proof}
By Lemma~\ref{lem:one} we can write
\begin{eqnarray*}
hf^{-\alpha}\ex_{k}&\in&hf^{-\alpha}f^{\alpha-1}[\ex_{k-\alpha+1}]_{p-\alpha}^{+}\\
&=&hf^{-1}[\ex_{k-\alpha+1}]_{p-\alpha}^{+}\\
&=&hf^{-1}f[f^{-1}\ex_{k-\alpha+1}]_{p-\alpha}\\
&=&h[[\ex_{k-\alpha}]_{p-2}^{+}\cap h^{-1}\gamma h\ex_{k-\alpha}]_{p-\alpha}\\
&\subset&h[h^{-1}\gamma h\ex_{k-\alpha}]_{p-\alpha}\\
&\subset&h[h^{-1}\gamma h\ex_{k-\alpha}]_{0}\\
&=&hh^{-1}hh^{-1}\gamma h\ex_{k-\alpha}\\
&=&\gamma h\ex_{k-\alpha}
\end{eqnarray*}
where we used \eqref{eqn:formidable} and Lemma~\ref{lem:list}.
\end{proof}
\medskip\smallskip

\noindent {\bf Proof of Theorem~\ref{thm:main}.} Consider the
system~\eqref{eqn:formidable}. Given some $k\geq p$ and
$\sigma\geq 1$, suppose $h\ex_{k-\alpha}\in\Y_{\sigma}$ for all
$\alpha\in\{1,\,2,\,\ldots,\,p\}$. Then by Lemma~\ref{lem:two} we
have
\begin{eqnarray*}
hf^{-\alpha}\ex_{k}
&=&\gamma h\ex_{k-\alpha}\\
&\in&\gamma\Y_{\sigma}\\
&=&\Y_{\sigma-1}
\end{eqnarray*}
for all $\alpha\in\{1,\,2,\,\ldots,\,p\}$. By compatibility
therefore $h\ex_{k}\in\Y_{\sigma-1}$. Hence we established
\begin{eqnarray}\label{eqn:cete}
h\ex_{k-\alpha}\in\Y_{\sigma}\quad\forall
\alpha\in\{1,\,2,\,\ldots,\,p\}\implies
h\ex_{k}\in\Y_{\sigma-1}\,.
\end{eqnarray}
Now note that $\Y_{\sigma-1}\subset\Y_{\sigma}$ by definition and
$\Y_{r}=\Y^{q}$, where $r\geq 1$ satisfies
$\gamma^{r}\Y^{q}=\Y_{0}$ because $\gamma$ is a deadbeat
interconnection. In the light of these facts, \eqref{eqn:cete}
implies $h\ex_{\tau-\alpha}\in\Y_{1}$ for all
$\alpha\in\{1,\,2,\,\ldots,\,p\}$, where $\tau:=rp$. Evoking
Lemma~\ref{lem:two} once again we have
\begin{eqnarray*}
hf^{-\alpha}\ex_{\tau}
&=&\gamma h\ex_{\tau-\alpha}\\
&\in&\gamma\Y_{1}\\
&=&\Y_{0}\,.
\end{eqnarray*}
The interpretation of this for the array~\eqref{eqn:dbsync} is
\begin{eqnarray}\label{eqn:coffee}
hf^{-\alpha}x_{i}(\tau)=hf^{-\alpha}x_{j}(\tau)\qquad\forall\,i,\,j\in\{1,\,2,\,\ldots,\,q\}
\end{eqnarray}
for all $\alpha\in\{1,\,2,\,\ldots,\,p\}$. Note that
Assumption~\ref{assume:singleton} and the first property listed in
Lemma~\ref{lem:list} imply $x=h^{-1}hx \cap [x]_{p-2}^{+}$ for all
$x\in\X$. Then we can write
\begin{eqnarray*}
x
&=& h^{-1}hx \cap [x]_{p-2}^{+}\\
&=& h^{-1}hx \cap f[f^{-1}x]_{p-2}\\
&=& h^{-1}hx \cap f(h^{-1}hf^{-1}x \cap[f^{-1}x]_{p-3}^{+})\\
&=& h^{-1}hx \cap fh^{-1}hf^{-1}x \cap f[f^{-1}x]_{p-3}^{+} \\
&\vdots&\\
&=& h^{-1}hx \cap fh^{-1}hf^{-1}x \cap \cdots \cap
f^{p-1}h^{-1}hf^{1-p}x \cap f^{p-1}[f^{1-p}x]_{-1}^{+}\\
&=& h^{-1}hx \cap fh^{-1}hf^{-1}x \cap \cdots \cap
f^{p-1}h^{-1}hf^{1-p}x
\end{eqnarray*}
where for the last step we used the property $[x]_{-1}^{+}=\X$. By
\eqref{eqn:coffee} we can then write
\begin{eqnarray*}
f^{-1}x_{i}(\tau)
&=& \bigcap_{\alpha\in\{1,\,2,\,\ldots,\,p\}}
f^{\alpha-1}h^{-1}hf^{-\alpha}x_{i}(\tau)\\
&=& \bigcap_{\alpha\in\{1,\,2,\,\ldots,\,p\}}
f^{\alpha-1}h^{-1}hf^{-\alpha}x_{j}(\tau)\\
&=& f^{-1}x_{j}(\tau)\,.
\end{eqnarray*}
Recall that $f$ is bijective. Hence $x_{i}(\tau)=x_{j}(\tau)$.
This equality emerges from arbitrary initial conditions. The
time-invariance of the array~\eqref{eqn:dbsync} therefore implies
that $x_{i}(k)=x_{j}(k)$ for all $k\geq\tau$ and all $i,\,j$.
\hfill\null\hfill$\blacksquare$

\begin{remark}
Note that the map $\ex\mapsto f([\ex]_{p-2}^{+}\cap h^{-1}\gamma
h\ex)$ is itself a deadbeat interconnection (over $\X^{q}$) under
the assumptions of Theorem~\ref{thm:main}.
\end{remark}

\begin{remark}
Observe that Theorem~\ref{thm:dbobs} directly follows from
Theorem~\ref{thm:main} by letting the interconnection
$\gamma:\Y^{2}\to\Y^{2}$ be such that
$\gamma_{1}\vay=\gamma_{2}\vay=y_{1}$.
\end{remark}

\section{An example}~\label{sec:dbex}

As an illustration of Theorem~\ref{thm:main}, we now provide an
example where an array of nonlinear observers achieves deadbeat
synchronization. The below pair $(f,\,h)$ is borrowed from
\cite{tuna12}.
\begin{eqnarray}\label{eqn:hmgsys}
f(x)=\left[\begin{array}{c}b\\c^{1/3}\\a^{3}+b^{3}\end{array}
\right]\,,\quad h(x)=a
\end{eqnarray}
with $(a,\,b,\,c)=x\in\X=\Real^{3}$ and $\Y=\Real$. The map $f$ is
bijective, hence Assumption~\ref{assume:bijective} holds, and we
have
\begin{eqnarray*}
[\xhat]_{1}^{+}\cap h^{-1}(y)=\left[\begin{array}{c}y\\{\hat
b}\\{\hat c}-{\hat a}^{3}+y^{3}\end{array} \right]
\end{eqnarray*}
which is singleton. (We refer the reader to \cite{tuna12} for
derivations.) Therefore Assumption~\ref{assume:singleton} is
satisfied with deadbeat horizon $p=3$. Regarding the deadbeat
interconnection $\gamma$, our choice is the one given in
\eqref{eqn:hmgcoup} with $[\gamma_{ij}]\in\Real^{q\times q}$ a
deadbeat coupling matrix. Hence Assumption~\ref{assume:interconn}
also holds. The question now is whether the triple
$(f,\,h,\,\gamma)$ is compatible or not.

\begin{claim}
The triple $(f,\,h,\,\gamma)$ described by \eqref{eqn:hmgsys} and
\eqref{eqn:hmgcoup} is compatible.
\end{claim}

\begin{proof}
Let $G:=[\gamma_{ij}]$. Since $G$ is a deadbeat coupling matrix,
for some integer $r\leq q-1$, we have $G^{r}=G^{r+1}=\one\ell^{T}$
where $\ell\in\Real^{q}$ is the left eigenvector of $G$ for the
eigenvalue $\lambda=1$ satisfying $\ell^{T}\one=1$. Let
$\vay=(y_{1},\,y_{2},\,\ldots,\,y_{q})\in\Real^{q}$ and
$\vay^{u}:=(y_{1}^{u},\,y_{2}^{u},\,\ldots,\,y_{q}^{u})$ for
$u\in\Rational$. Then we have
$\Y_{\sigma}=\{\vay\in\Y^{q}:\vay^{3}\in\N(G^{\sigma}-\one\ell^{T})\}$.
Note that
\begin{eqnarray}\label{eqn:highhopes}
\vay,\,\vi\in\Y_{\sigma}&\implies&
\vay^{3},\,\vi^{3}\in\N(G^{\sigma}-\one\ell^{T})\nonumber\\
&\implies& \vay^{3}+\vi^{3}\in\N(G^{\sigma}-\one\ell^{T})\nonumber\\
&\implies& (\vay^{3}+\vi^{3})^{1/3}\in\Y_{\sigma}
\end{eqnarray}
since $\N(G^{\sigma}-\one\ell^{T})$ is a linear subspace.

Given $\ex=(x_{1},\,x_{2},\,\ldots,\,x_{q})$ with
$x_{i}=(a_{i},\,b_{i},\,c_{i})$ note that
\begin{eqnarray*}
f(x_{i})=\left[\begin{array}{c}b_{i}\\ c_{i}^{1/3} \\
a_{i}^{3}+b_{i}^{3} \end{array}
\right],\,\, f^{2}(x_{i})=\left[\begin{array}{c}c_{i}^{1/3}\\
(a_{i}^{3}+b_{i}^{3})^{1/3} \\ *
\end{array} \right],\,\, f^{3}(x_{i})=\left[\begin{array}{c}(a_{i}^{3}+b_{i}^{3})^{1/3}\\ * \\ * \end{array}
\right]
\end{eqnarray*}
Suppose now $h(\ex),\,h(f(\ex)),\,h(f^{2}(\ex))\in\Y_{\sigma}$ for
some $\sigma$.
\begin{eqnarray*}
\left[\begin{array}{c}a_{1}\\ a_{2} \\\vdots
\\ a_{q}
\end{array}
\right]\,,\quad \left[\begin{array}{c}b_{1}\\ b_{2} \\\vdots
\\ b_{q}
\end{array}
\right]\,,\quad \left[\begin{array}{c}c_{1}^{1/3}\\ c_{2}^{1/3}
\\ \vdots
\\ c_{q}^{1/3}
\end{array}
\right]\in\Y_{\sigma}\,.
\end{eqnarray*}
Then by \eqref{eqn:highhopes} we can write
\begin{eqnarray*}
h(f^{3}(\ex))=\left[\begin{array}{c}(a_{1}^{3}+b_{1}^{3})^{1/3}\\
(a_{2}^{3}+b_{2}^{3})^{1/3} \\\vdots
\\ (a_{q}^{3}+b_{q}^{3})^{1/3}
\end{array}
\right] = \left(\left[\begin{array}{c}a_{1}\\ a_{2} \\\vdots
\\ a_{q}
\end{array}
\right]^{3}+\left[\begin{array}{c}b_{1}\\ b_{2} \\\vdots
\\ b_{q}
\end{array}
\right]^{3}\right)^{1/3} \in\Y_{\sigma}\,.
\end{eqnarray*}
Hence the result.
\end{proof}
\medskip\smallskip

Now that all the required conditions are met, the following array
of $q$ coupled deadbeat observers should achieve deadbeat
synchronization
\begin{eqnarray*}
x_{i}^{+}&=&f\left([x_{i}]_{1}^{+}\cap
h^{-1}\left(\left(\gamma_{i1}a_{1}^{3}+\gamma_{i2}a_{2}^{3}+\ldots+\gamma_{iq}a_{q}^{3}\right)^{1/3}\right)\right)\\
&=&\left[\begin{array}{c} b_{i} \\
\left(c_{i}-a_{i}^{3}+\gamma_{i1}a_{1}^{3}+\gamma_{i2}a_{2}^{3}+\ldots+\gamma_{iq}a_{q}^{3}\right)^{1/3}\\
b_{i}^{3}+\gamma_{i1}a_{1}^{3}+\gamma_{i2}a_{2}^{3}+\ldots+\gamma_{iq}a_{q}^{3}
\end{array} \right]
\end{eqnarray*}
for all deadbeat coupling matrices $[\gamma_{ij}]\in\Real^{q\times
q}$. For instance, letting $[\gamma_{ij}]$ be as in
\eqref{eqn:dbcoup}, the following array is obtained.
\begin{eqnarray*}
x_{1}^{+}
&=&\left[\begin{array}{c} b_{1} \\
\left(c_{1}-a_{1}^{3}+0.4a_{1}^{3}-0.2a_{2}^{3}+3.2a_{3}^{3}-2.4a_{4}^{3}\right)^{1/3}\\
b_{1}^{3}+0.4a_{1}^{3}-0.2a_{2}^{3}+3.2a_{3}^{3}-2.4a_{4}^{3}
\end{array} \right]\\
x_{2}^{+}
&=&\left[\begin{array}{c} b_{2} \\
\left(c_{2}-a_{2}^{3}+0.4a_{1}^{3}-0.2a_{2}^{3}+0.2a_{3}^{3}+0.6a_{4}^{3}\right)^{1/3}\\
b_{2}^{3}+0.4a_{1}^{3}-0.2a_{2}^{3}+0.2a_{3}^{3}+0.6a_{4}^{3}
\end{array} \right]\\
%\end{eqnarray*}
%\begin{eqnarray*}
x_{3}^{+}
&=&\left[\begin{array}{c} b_{3} \\
\left(c_{3}-a_{3}^{3}+0.2a_{1}^{3}-0.6a_{2}^{3}+0.6a_{3}^{3}+0.8a_{4}^{3}\right)^{1/3}\\
b_{3}^{3}+0.2a_{1}^{3}-0.6a_{2}^{3}+0.6a_{3}^{3}+0.8a_{4}^{3}
\end{array} \right]\\
x_{4}^{+}&=&\left[\begin{array}{c} b_{4} \\
\left(c_{4}-a_{4}^{3}+0.3a_{1}^{3}-0.4a_{2}^{3}+0.9a_{3}^{3}+0.2a_{4}^{3}\right)^{1/3}\\
b_{4}^{3}+0.3a_{1}^{3}-0.4a_{2}^{3}+0.9a_{3}^{3}+0.2a_{4}^{3}
\end{array} \right]
\end{eqnarray*}
The above array achieves synchronization in $pr=9$ steps, where
$r=3$ is the smallest integer satisfying
$[\gamma_{ij}]^{r+1}=[\gamma_{ij}]^{r}$. Fig.~\ref{fig:dbsync}
shows the simulation results for the initial conditions
$x_{1}(0)=(0.5,\,0.5,\,0.5)$, $x_{2}(0)=(0,\,-1,\,0)$,
$x_{3}(0)=(-0.5,\,0,\,-0.5)$, and $x_{4}(0)=(-1,\,0.5,\,0)$.
\begin{figure}[h]
\begin{center}
\includegraphics[scale=0.6]{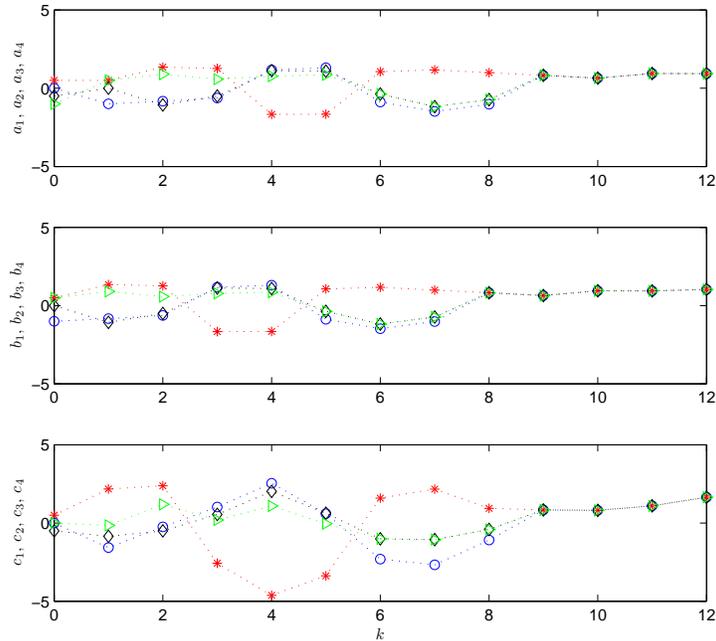}\caption{Evolution of states $x_{1}$ (marked $*$), $x_{2}$ (marked
$\circ$), $x_{3}$ (marked $\diamond$), and $x_{4}$ (marked
$\triangleright$).}\label{fig:dbsync}
\end{center}
\end{figure}

\section{Notes}\label{sec:notes}

The deadbeat interconnection considered in this paper is fixed. A
possible relaxation is suggested by the proof of
Theorem~\ref{thm:main}. Namely, deadbeat synchronization would
still be achieved with time-varying interconnection
$\gamma(k,\,\vay)$ provided that the sets $\Y_{\sigma}$ stayed
fixed and the relation
$\gamma(k,\,\Y_{\sigma})\subset\Y_{\sigma-1}$ was satisfied at all
times $k$. This is very closely related to what we mentioned in
Remark~\ref{rem:timeinvariance} regarding the linear
array~\eqref{eqn:moretempting}. Further generalization in this
direction seems nevertheless not to be an easy task.

A practical design problem is how to construct a deadbeat
interconnection compatible with a given $(f,\,h)$ pair. A
primitive solution to this problem is to select an interconnection
that admits a connected graph that is a (directed) tree. In that
case, in an array of $q$ systems, $q-1$ of the systems would each
be driven by exactly one other system, i.e., for each
$i\in\{1,\,2,\,\ldots,\,q-1\}$ we would have
$\gamma_{i}(\vay)=y_{j}$ for some $j\neq i$, and one system (the
root) would be driven by no one, i.e., $\gamma_{q}(\vay)=y_{q}$.
However, when one starts considering interconnection schemes that
include cycles in their graphs, the problem seems to lack an
obvious systematic solution.

\bibliographystyle{IEEEtran}
\bibliography{references}
\end{document}